\documentclass{amsart}

\usepackage[utf8]{inputenc}
\usepackage[T1]{fontenc}
\usepackage[english]{babel}

\usepackage{amsmath, amssymb, amsthm}
\usepackage{enumitem}
\usepackage{hyperref}
\usepackage{float}
\usepackage{tikz}
\usepackage{caption, subcaption}
\usepackage[capitalise]{cleveref}

\newcommand{\arrow}{\draw[->]}
\newcommand{\bluearrow}{\draw[blue, ->]}
\newcommand{\redarrow}{\draw[red, ->]}

\newcommand{\Z}{\mathbb{Z}}
\newcommand{\D}{\mathcal{D}}
\newcommand{\T}{\mathbb{T}}
\newcommand{\OO}{\mathcal{O}}

\newtheorem{proposition}{Proposition}[section]
\newtheorem{lemma}[proposition]{Lemma}
\newtheorem{theorem}[proposition]{Theorem}
\newtheorem{corollary}[proposition]{Corollary}

\theoremstyle{definition}
\newtheorem{remark}[proposition]{Remark}
\newtheorem{definition}[proposition]{Definition}
\newtheorem{example}[proposition]{Example}

\title{On strong shift equivalence for row-finite graphs and C*-algebras}
\author[K.A. Brix]{Kevin Aguyar Brix}
\address[K.A. Brix]{School of Mathematics and Statistics, University of Glasgow, University Place, Glasgow G12 8QQ, United Kingdom}
\email{kabrix.math@fastmail.com}

\author[P. Gautam]{Pete Gautam}
\address[P. Gautam]{School of Mathematics and Statistics, University of Glasgow, University Place, Glasgow G12 8QQ, United Kingdom}
\email{2481471G@student.gla.ac.uk}

\makeatletter
\@namedef{subjclassname@2020}{%
  \textup{2020} Mathematics Subject Classification}
\makeatother
\thanks{This work is the result of a summer project at the University of Glasgow. 
It was funded by the London Mathematical Society through an Undergraduate Research Bursary 2022 (URB--2022--42) and the School of Mathematics and Statistics at the University of Glasgow.
The first-named author was supported by the Independent Research Fund Denmark (Case number 1025-00004B)}
\begin{document}

    \begin{abstract}
     This note extends and strengthens a theorem of Bates that says that row-finite graphs that are strong shift equivalent have Morita equivalent graph C*-algebras.
  This allows us to ask whether our stronger notion of Morita equivalence does in fact characterise strong shift equivalence. 
  We believe this will be relevant for future research on infinite graphs and their C*-algebras.
  We also study insplits and outsplits as particular examples of strong shift equivalences 
  and show that the induced Morita equivalences respect a whole family of weighted gauge actions.
  We then ask whether strong shift equivalence is generated by (generalised) insplits and outsplits.
     \end{abstract}

    \maketitle

    \section{Introduction}
    \noindent Strong shift equivalence was originally defined by Williams~\cite{Williams} in his fundamental work on conjugacy of shifts of finite type. 
Shifts of finite type are represented by finite directed graphs or, equivalently, by the adjacency matrices of the graphs~\cite{Lind-Marcus}.
A pair of such finite square non-negative integer matrices $A$ and $B$ are elementary strong shift equivalent if there are rectangular non-negative integer matrices $R$ and $S$ such that $A = R S$ and $S R = B$;
and $A$ and $B$ are \emph{strong shift equivalent}, if there are $A = A_0, A_1,\dots,A_n = B$ such that $A_{i-1}$ is elementary strong shift equivalent to $A_i$ for all $i=1,\dots,n$.
Williams' theorem is then that a pair of shifts of finite type $X_A$ and $X_B$ are conjugate precisely when the adjacency matrices $A$ and $B$ of the corresponding graphs are strong shift equivalent.
If $E$ and $F$ are the graphs of $A$ and $B$, then the proof of this landmark result uses an auxiliary bipartite graph $E_3$ that is built from $E$ and $F$.
It is therefore natural to define a notion of strong shift equivalence for general directed graphs using such an auxiliary graph,
and this is the notion that we will explore in this paper.

In 1980, Cuntz and Krieger constructed C*-algebras $\OO_A$ from infinite irreducible shifts of finite type determined by the matrix $A$~\cite{Cuntz-Krieger}.
They built a large class of simple and purely infinite C*-algebras such that the full $n$-shift produced the Cuntz algebras $\OO_n$,
and today we refer to them as \emph{Cuntz--Krieger algebras}.
This construction was later generalised to C*-algebras from general and infinite graphs~\cite{raeburn2005graph},
and the advantage was a lot of the structure of the otherwise very complicated C*-algebras could be read off of the graphs immediately.
Moreover, the graph C*-algebras carry important additional structure, such as a canonical gauge action (an action of the complex unit circle)
and a commutative subalgebra usually called the diagonal subalgebra. 

In this note, we revisit the notion of strong shift equivalence for row-finite graphs from~\cite{bates2002applications}
in which Bates proves that strong shift equivalent graphs have Morita equivalent graph C*-algebras.
We note that the assumption that the graphs be row-finite is essential, and there are counterexamples in~\cite[Example 5.4]{bates2002applications}.
We extend and strengthen this theorem to include the case when the graphs have sinks
and we show that not just the C*-algebra but the triple of the graph C*-algebra, its diagonal subalgebra, and its canonical gauge action are Morita equivalent,
when the graphs are strong shift equivalent.
Inspired by work of Matsumoto in the case of irreducible shifts of finite type~\cite{Matsumoto2017}, we may then ask the question of the converse: 
suppose a pair of row-finite graphs have Morita equivalent triples of graph C*-algebras, diagonal subalgebra, and canonical gauge action,
are they then strong shift equivalent? 

Next, we study state splittings (insplits and outsplits) which give particular examples of strong shift equivalences.
The state splittings also emerged in Williams' work mentioned above, and although insplits and outsplits have dual definitions,
insplits produce isomorphic $C^*$-algebras, while outsplits produce Morita equivalent $C^*$-algebras. 
For these particular examples of strong shift equivalences, we show that the Morita equivalence of graph C*-algebras
respects a whole family of weighted gauge actions (and not just the canonical gauge action).
We then ask the question of whether strong shift equivalence as an equivalence relations is generated by (generalised) insplits and outsplits.
The adjective 'generalised' seems to be necessary; examples of conjugate graphs for which no finite chain of standard splittings can connect them will appear in~\cite{ABCEW}.
For the case of shifts of finite type, this was proved by Williams.
These questions are natural from the point of view of previous work on dynamical relations on graphs and their C*-algebras,
the main work being that of Eilers and Ruiz~\cite{eilers2019refined}.
They only consider graphs with finitely many vertices, whereas we ask kindred questions for general row-finite graphs. 

In~\cref{sec:preliminaries} we introduce the relevant notation and preliminaries on graphs and their C*-algebras, in~\cref{sec:sse} we define strong shift equivalence for row-finite graphs possibly containing sinks and we prove our strenghthening of Bates' theorem. In~\cref{sec:insplits,sec:outsplits}, we study insplits and outsplits, respectively, show that they give examples of strong shift equivalences, and study how the induced Morita equivalence respects a whole family of weighted gauge actions.
It is an unfortunate fact that the literature on graph C*-algebras uses two dual kinds of notation.
We use the convention from~\cite{raeburn2005graph} which is incidentally opposite to that of~\cite{bates2002applications}.

    \section{Preliminaries} \label{sec:preliminaries}
    For notation and terminology related to directed graphs and graph C*-algebras we use \cite{raeburn2005graph}.
\subsection{Directed graphs}
    A \emph{directed graph} $E = (E^0, E^1, r, s)$ is a tuple of countable sets $E^0$ and $E^1$ and functions $r, s\colon E^1 \to E^0$. A $v \in E^0$ is called a \emph{vertex}, and an $e \in E^1$ is called an \emph{edge}. For $e \in E^1$ and $v \in E^0$, if $r(e) = v$, then $e$ has \emph{range} $v$, and if $s(e) = v$, then $e$ has \emph{source} $v$. For $v \in E^0$, we say that $v$ is a \emph{source} if the preimage $r^{-1}(v) = \varnothing$, and $v$ is a \emph{sink} if $s^{-1}(v) = \varnothing$. The graph $E$ is \emph{row-finite} if for all $v \in E_1^0$, $r^{-1}(v)$ is finite. In this paper, we will only be concerned with row-finite graphs.
    
    A \emph{finite path} on a graph $E$ is a sequence $e_1 \cdots e_n$ where $s(e_i) = r(e_{i+1}$ for all $i = 1 \cdots n-1$ or just a vertex $v\in E^0$; we let $|e_1\cdots e_n| = n$ denote the length of the path, and if $v$ is a vertex, we put $|v| = 0$. We let $E^*$ be the collection of all finite paths and the range and source maps naturally extend to $E^*$. Two finite paths $e_1 \cdots e_n$ and $f_1 \cdots f_m$ can be concatenated precisely if $s(e_n) = r(f_1)$. 
    
\subsection{Graph C*-algebras}
    Let $E$ be a row-finite directed graph. 
    A \emph{Cuntz--Krieger $E$-family} is a collection of partial isometries $\{S_e : e\in E^1\}$ with orthogonal ranges and a collection of projections $\{P_v : v\in E^0\}$ satisfying the \emph{Cuntz--Krieger relations}:
        \begin{enumerate}[label=\textbf{(CK\arabic*)}]
            \item for all $e \in E^1$, $S_e^* S_e = P_{s(e)}$; and
            \item for all $v \in E^0$ not a source, we have
            \[P_v = \sum_{e \in r^{-1}(v)} S_e S_e^*.\]
        \end{enumerate}
    For convenience, We simply write $\{S,P\}$ for the $E$-family $\{S_e,P_v: e\in E^1, v\in E^0\}$.
    The \emph{graph C*-algebra} $C^*(E)$ of $E$ is the universal C*-algebra generated by a Cuntz--Krieger $E$-family. This means that $C^*(E)$ is generated by an $E$-family $\{s,p\}$ and that if $\{T,Q\}$ is any $E$-family in some C*-algebra $B$, then there exists a *-homomorphism $\pi\colon C^*(E) \to B$ satisfying $\pi(s_e) = T_e$ for all $e\in E^1$ and $\pi(p_v) = Q_v$ for all $v\in E^0$.
    We usually use capital letters for Cuntz--Krieger families $\{S,P\}$, but small letters $\{s,p\}$ to emphasise that the family is universal in the above sense.
    
    The graph C*-algebra $C^*(E)$ admits an action $\gamma^E$ of the circle group $\T$ given as follows:
    for $e\in E^1$ we have $\gamma^E_z(s_e) = z s_e$ for $z\in \T$ and for $v\in E^0$ we have $\gamma^E(p_v) = p_v$.
    We refer to $\gamma^E$ as the \emph{canonical gauge action}.
    
    Let $f\colon E^1 \to \Z$ be any function.
    We call this an \emph{edge function}.
    We can extend it to a function $f\colon E^*\to \Z$ by putting $f(e_1 \cdots e_n) = f(e_1) + \cdots + f(e_n)$ for every finite path $e_1 \cdots e_n$ in $E^*$ and $f(v) = 0$ for every vertex $v$ in $E^0$. 
    By the universal property of the graph C*-algebra, we may now define a \emph{weighted gauge action} $\gamma^{E,f}$ of $\T$ on $C^*(E)$ as follows:
    for $e\in E^1$ and $v\in E^0$ we let
    \[
        \gamma_z^{E,f}(s_e) = z^{f(e)} s_e, \quad \textrm{and}\quad 
        \gamma_z^{E,f}(p_v) = p_v,
    \]
    for every $z\in \T$.
    We refer to $f$ as the \emph{edge function} of $\gamma^{E,f}$.
    Note that the canonical gauge action has edge function $f=1$.
    
    \subsection{Morita equivalence}
    Two graph C*-algebras $C^*(E_1)$ and $C^*(E_2)$ are \emph{Morita equivalent} if there exist a C*-algebra $C$ and complementary full projections $P_1$ and $P_2$ in $M(C)$, the multiplier algebra of $C$, and *-isomorphisms $\pi_1\colon C^*(E_1) \to P_1 C P_1$ and $\pi_2\colon C^*(E_2) \to P_2 C P_2$.
    In this paper, we also consider actions of the circle group $\T$ on graph C*-algebras.
    If $\alpha\colon \T\curvearrowright C^*(E_1)$ and $\beta\colon \T\curvearrowright C^*(E_2)$ are actions, then $(C^*(E_1), \alpha)$ and $(C^*(E_2), \beta)$ are \emph{equivariantly Morita equivalent}
    if, in addition, $C$ carries an action $\gamma$ of $\T$ such that $\pi_1\circ \alpha = \gamma \circ \pi_1$ and $\pi_2\circ \beta = \gamma\circ \pi_2$, cf. \cite[Section 4]{Combes}.
    
    The graph C*-algebras carry distinguished diagonal subalgebras,
    and we say that the pairs $(C^*(E_1),\D(E_1))$ and $(C^*(E_2), \D(E_2))$ are \emph{Morita equivalent} if the C*-algebra $C$ above carries a distinguished subalgebra $\D(C)$ such that $\pi_1(\D(E_1)) = P_1 \D(C)$ and $\pi_2(\D(E_2)) = P_2 \D(C)$.
    
    Combining the two conditions, we simply say that the triples 
    \[
    (C^*(E_1), \gamma^{E_1}, \D(E_1))\quad \textrm{and}\quad (C^*(E_2), \gamma^{E_2}, \D(E_2))
    \]
    are Morita equivalent if the *-isomorphisms $\pi_1$ and $\pi_2$ are both equivariant and diagonal-preserving.

    \section{Strong Shift Equivalence} \label{sec:sse}
    In this section, we will explore the concept of strong shift equivalence, a graph theoretical property with origins in symbolic dynamics. 
    Improving a theorem of Bates, we find that two graphs that are strong shift equivalent have Morita equivalent triples of graph C*-algebra, gauge action, and diagonal subalgebra.
    This gives us two *-isomorphisms, which we will then use to show equivariance with edge functions with a general result, and apply it to two specific cases that we will use in the later sections.
    
    The definition below is from~\cite[Section 5]{bates2002applications} (cf.~\cite{Ashton}).
    
    \begin{definition}[Strong shift equivalence] 
        \label{def:sse}
        Let $E_1$ and $E_2$ be row-finite graphs. We say that $E_1$ and $E_2$ are \emph{elementary strong shift equivalent} if there exists a row-finite graph $E_3$ such that:
        \begin{itemize}
            \item $E_3^0 = E_1^0 \cup E_2^0$ and $E_1^0 \cap E_2^0 = \varnothing$;
            \item $E_3^1 = E_{21}^1 \cup E_{12}^1$, where $E_{21}^1$ consists of edges with source in $E_1^0$ and range in $E_2^0$, and $E_{12}^1$ vice versa;
            \item there exist source- and range-preserving bijections $\theta_1\colon E_1^1 \to E_3^2(E_1^0, E_1^0)$ and $\theta_2\colon E_2^1 \to E_3^2(E_2^0, E_2^0)$, where $E_3^2(A, B)$ are paths in $E_3$ of length 2 with source $A \subseteq E_3^0$ and range $B \subseteq E_3^0$;
            \item if $v \in E_3^0$ is a source, then there exists some $\eta \in E_3^1$ such that $s_3^{-1}(v) = \{\eta\}$ and $r_3^{-1}(r_3(\eta)) = \{\eta\}$.
        \end{itemize}
        If so, we say that \emph{$E_1$ and $E_2$ are elementary SSE via $E_3$}.
        Strong shift equivalence is then the equivalence relation on row-finite directed graphs generated by elementary strong shift equivalence.
    \end{definition}
    
    \begin{remark}
        The first three conditions for SSE are given in~\cite[Definition~5.1]{bates2002applications}.
        The final condition allows us to ensure that we still have Morita equivalence between $C^*(E_1)$ and $C^*(E_2)$ (see the result below) when we include graphs with sources. 
        Note that this is slightly different from what Bates suggests in~\cite[Remark 5.3]{bates2002applications}.
    \end{remark}
    
    \begin{example}
        Let $E_1$ and $E_2$ be the following graphs:
        
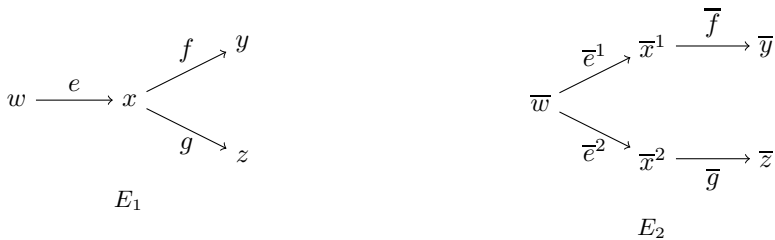
\begin{figure}[H]
            \centering
            \begin{subfigure}{0.45\textwidth}
                \centering
                \begin{tikzpicture}
                    \node (w) at (0, 0) {$w$};
                    \node (x) at (1.5, 0) {$x$};
                    \node (y) at (3, 0.75) {$y$};
                    \node (z) at (3, -0.75) {$z$};
    
                    \arrow (w) -- node[above] {$e$} (x);
                    \arrow (x) -- node[above] {$f$} (y);
                    \arrow (x) -- node[below] {$g$} (z);
                \end{tikzpicture}
                \caption{$E_1$}
            \end{subfigure}
            \hfill
            \begin{subfigure}{0.45\textwidth}
                \centering
                \begin{tikzpicture}
                    \node (w) at (0, 0) {$\overline{w}$};
                    \node (x1) at (1.5, 0.75) {$\overline{x}^1$};
                    \node (x2) at (1.5, -0.75) {$\overline{x}^2$};
                    \node (y) at (3, 0.75) {$\overline{y}$};
                    \node (z) at (3, -0.75) {$\overline{z}$};
    
                    \arrow (w) -- node[above] {$\overline{e}^1$} (x1);
                    \arrow (w) -- node[below] {$\overline{e}^2$} (x2);
                    \arrow (x1) -- node[above] {$\overline{f}$} (y);
                    \arrow (x2) -- node[below] {$\overline{g}$} (z);
                \end{tikzpicture}
                \caption{$E_2$}
            \end{subfigure}
        \end{figure}
        \noindent Note that both $E_1$ and $E_2$ have sources ($w$ in $E_1$ and $\overline{w}$ in $E_2$). Moreover, $E_1$ and $E_2$ are SSE by the following graph $E_3$:
        
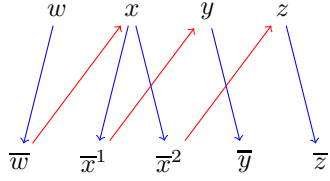
\begin{figure}[H]
            \centering
            \begin{tikzpicture}
                \node (w) at (0, 2) {$w$};
                \node (x) at (1, 2) {$x$};
                \node (y) at (2, 2) {$y$};
                \node (z) at (3, 2) {$z$};

                \node (W) at (-0.5, 0) {$\overline{w}$};
                \node (X1) at (0.5, 0) {$\overline{x}^1$};
                \node (X2) at (1.5, 0) {$\overline{x}^2$};
                \node (Y) at (2.5, 0) {$\overline{y}$};
                \node (Z) at (3.5, 0) {$\overline{z}$};

                \bluearrow (w) -- (W);
                \bluearrow (x) -- (X1);
                \bluearrow (x) -- (X2);
                \bluearrow (y) -- (Y);
                \bluearrow (z) -- (Z);

                \redarrow (W) -- (x);
                \redarrow (X1) -- (y);
                \redarrow (X2) -- (z);
            \end{tikzpicture}
        \end{figure}
        \noindent The blue arrows are in $E_{21}^1$ and the red arrows are in $E_{12}^1$.
    \end{example}

    \begin{theorem} \label{thm:morita-equivalence-sse}
        Let $E_1$ and $E_2$ be row-finite directed graphs that are SSE via $E_3$. Then there exist complementary full projections 
        \[P_1 = \sum_{v \in E_1^0} r_v, \qquad P_2 = \sum_{v \in E_2^0} r_v\]
        in $M(C^*(E_3))$, an action $\alpha$ of $\T$ on $C^*(E_3)$, and *-isomorphisms $\pi_1\colon C^*(E_1) \to P_1C^*(E_3)P_1$ that is $\gamma^{E_1}-\alpha$-equivariant and $\pi_2\colon C^*(E_2) \to P_2C^*(E_3)P_2$ that is $\gamma^{E_2}-\alpha$-equivariant and such that $\pi(\D(E_1)) = P_1 \D(E_3) P_1$ and $\pi(\D(E_2)) = P_2 \D(E_3) P_2$.
        In particular, the triples $(C^*(E_1), \gamma^{E_1}, \D(E_1))$ and $(C^*(E_2), \gamma^{E_2}, \D(E_2))$ are Morita equivalent. 
    \end{theorem}
    
    \begin{proof}
        We follow the proof of \cite[Theorem~5.2]{bates2002applications}.
        Bates proves that the graph $C^*$-algebras are Morita equivalence assuming the graphs have no sources. 
        Below we emphasise the parts of the proof that are new. 
        
        Suppose $E_1$ and $E_2$ are SSE via $E_3$ and choose bijections $\theta_1$ and $\theta_2$ in accordance with~\cref{def:sse}.
        Let $\{s,p\}$ and $\{u,r\}$ be the canonical Cuntz--Krieger families in $C^*(E_1)$ and $C^*(E_3)$, respectively.
        Define a Cuntz--Krieger $E_1$ family $\{S,P\}$ in $C^*(E_3)$ as follows:
        for $e\in E_1^1$ let $S_e = u_{\theta_1(e)}$, and for $v\in E_1^0$ let $P_v = r_v$.
        The proof that this is a Cuntz--Krieger $E_1$ family is similar to Bates' proof. 
        By the universal property of graph C*-algebras, there exists a *-homomorphism $\pi_1\colon C^*(E_1) \to C^*(E_3)$ given by $\pi_1(s_e) = S_e$ for $e\in E_1^1$ and $\pi_1(p_v) = P_v$ for $v\in E_1^0$.
        
        Note that we have $P_1P_vP_1 = P_v$ for all $v \in E_1^0$, and for all $e \in E_1^1$ we have $P_1S_eP_1 = S_e$. Since $\{S, P\}$ generates $C^*(S, P)$, we find that $C^*(S, P) \subseteq P_1C^*(E_3)P_1$.
        We claim that this is in fact an equality.
        
        Let $\mu, \nu \in E_3^*$. 
        By \cite[Lemma~2.10]{raeburn2005graph}, we know that
        \[P_1 u_\mu u_\nu^* P_1 = \begin{cases}
            u_\mu u_\nu^* & r_3(\mu), r_3(\nu) \in E_1^0 \\
            0 & \textrm{otherwise}.
        \end{cases}\]
        If $u_\mu u_\nu^* \neq 0$, then we also have $s_3(\mu) = s_3(\nu)$. So, if $P_1s_\mu s_\nu^*P_1 \neq 0$, then we find that either both $\mu, \nu$ have even length, or $\mu, \nu$ have odd length. If $\mu, \nu$ have even length, then $\mu = \alpha_1 \dots \alpha_m$ and $\beta = \beta_1 \dots \beta_n$, with $\alpha_i, \beta_j \in E^2_3(E_1^0, E_1^0)$ for $1 \leq i \leq m$ and $1 \leq j \leq n$. So, 
        \begin{align*}
            P_1u_\mu u_\nu^*P_1 = u_\mu u_\nu^* &= u_{\alpha_1} \dots u_{\alpha_n} u_{\beta_m}^* \dots u_{\beta_1}^* \\
            &= S_{\theta_1^{-1}(\alpha_1)} \dots S_{\theta_1^{-1}(\alpha_n)} S_{\beta_m}^* \dots S_{\beta_1}^* \in C^*(S, P).
        \end{align*}
        Instead, we can have $\mu, \nu$ both of odd length. If $s_3(\mu) = s_3(\nu)$ is not a source (in $E_3$), then we know that 
        \[r_{s_3(\mu)} = \sum_{\eta \in r_3^{-1}(s_3(\mu))} u_\eta u_\eta^*.\]
        This implies that
        \[u_\mu u_\nu^* = u_\mu r_{s_3(\mu)} u_\nu^* = \sum_{\eta \in r_3^{-1}(s_3(\mu))} u_{\mu \eta} u_{\nu \eta}^*.\]
        Since $\mu \eta$ and $\nu \eta$ are paths of even length starting and ending at $E_1^0$ for all $\eta \in r_3^{-1}(s_3(\mu))$, we again find that $u_\mu u_\nu^* \in C^*(S, P)$. Otherwise, we know that $s_3(\mu)$ is a source (in $E_3$). In that case, $s_3^{-1}(\mu_1) = \{\eta\}$, with $r_3^{-1}(r_3(\eta)) = \{\eta\}$, for some $\eta \in E_3^1$. Now, denote $\mu = \mu' \eta$ and $\nu = \nu' \eta$. Then
        \[u_\mu u_\nu^* = u_{\mu'} u_\eta u_\eta^* u_{\nu'}^* = u_{\mu'} r_{r_3(\eta)} u_{\nu'}^* = u_{\mu'} r_{s_3(\mu)} u_{\nu'}^* = u_{\mu'} u_{\nu'}^*.\]
        We know that $\mu'$ and $\nu'$ are paths of even length starting and ending at $E_1^0$, so we find that $u_\mu u_\nu^* \in C^*(S, P)$. 
        Thus we see that $u_\mu u_\nu^* \in C^*(S, P)$, for all $\mu, \nu \in E_3^*$. Hence, the image of $\pi_1$ is $C^*(S, P) = P_1C^*(E_3)P_1$.
        
        Next, we show that the map $\pi_1$ is injective. 
        Define a continuous $\T$-action $\alpha$ on $C^*(E_3)$ by
        \[\alpha_z(u_{\eta}) =  \begin{cases}
            z^{1/2} u_\eta & \textrm{if}~\eta \in E_{21}^1 \\
            z^{1/2} u_\eta & \textrm{if}~\eta \in E_{12}^1,
        \end{cases}\]
        and $\alpha_z(r_v) = r_v$. Now, for $z \in \T$ and $e \in E_1^1$, we find that 
        \[
            \alpha_z(S_e) = \alpha_z(u_{\theta_1(e)}) = \alpha_z(u_{\eta_2}) \alpha_z(u_{\eta_1}) = z u_{\eta_2} u_{\eta_1} = z S_e,
        \]
        where $\theta_1(e) = \eta_2 \eta_1$, and for all $v \in E_1^0$, $\alpha_z(P_v) = \alpha_z(r_v) = r_v = P_v$. 
        Note that $\alpha$ fixes each partial sum of $P$, and so $P$ remains fixed on the extension to $M(C^*(E_3))$.
        This means that $\alpha$ restricts to an action on $P_1 C^*(E_3) P_1 = C^*(S,P)$ that satisfies the hypotheses of the gauge-invariant uniqueness theorem~\cite[Theorem 2.2]{raeburn2005graph}.
        This shows that $\pi_1$ is injective.
    
        We now show that $\pi_1(\D(E_1)) = P_1\D(E_3)P_1$. Let $\alpha \in E_1^*$. Then
        \begin{align*}
            S_\alpha S_\alpha^* &= S_{\alpha_n} \dots S_{\alpha_1} S_{\alpha_1}^* \dots S_{\alpha_{n-1}}^* S_{\alpha_n}^* \\
            &= u_{\theta_1(\alpha_n)} \dots u_{\theta_1(\alpha_1)} u_{\theta_1(\alpha_1)}^* \dots u_{\theta_1(\alpha_n)}^* = u_\beta u_\beta^*,
        \end{align*}
        with $\beta = \theta_1(\alpha_n) \theta_1(\alpha_{n-1}) \dots \theta_1(\alpha_1)$. We know that $u_{\theta_1(\alpha_i)} \in P_1C^*(E_3)P_1$ for all $1 \leq i \leq |\alpha|$, so we find that $S_\alpha S_\alpha^* = u_\beta u_\beta^* \in P_1\D(E_3)P_1$. So, $\pi(\D(E_1)) \subseteq P_1\D(E_3)P_1$.
        
        Next, let $\beta \in E_3^*$ such that $u_\beta u_\beta^* \in P_1\D(E_3)P_1$. We must have $P_1u_\beta u_\beta^*P_1 = u_\beta u_\beta^*$, so $r_3(\beta) \in E_1^0$. If $|\beta|$ is even, then we can denote $\beta = \beta_{n} \beta_{n-1} \dots \beta_1$, with $|\beta_i| = 2$ and $\beta_i \in E^1_3(E_1^0, E_1^0)$ for $1 \leq i \leq n$. So, we have $S_{\theta_1^{-1}(\beta_i)} = u_{\beta_i}$ for $1 \leq i \leq |\beta|$. Hence, 
        \[u_\beta u_\beta^* = S_\alpha S_\alpha^*,\]
        with $\alpha = \theta_1^{-1}(\beta_n) \theta^{-1}(\beta_{n-1}) \dots \theta^{-1}(\beta_1)$. Next, assume that $|\beta|$ is odd. If $s_3(\beta)$ is not a source (in $E_3$), then
        \[u_\beta u_\beta^* = u_\beta p_{s_3(\beta)} u_\beta^* = \sum_{\eta \in r_3^{-1}(s_3(\beta))} u_{\beta \eta} u_{\beta \eta}^*.\]
        Since $r_3(\beta \eta) = r_3(\beta) \in E_1^0$ and $|\beta \eta|$ is even, the argument above shows that $u_{\beta \eta} u_{\beta \eta}^* \in \pi(\D(E_1))$ for all $\eta \in r_3^{-1}(s(\beta))$. Hence, $u_\beta u_\beta^* \in \pi(\D(E_1))$.
        Otherwise, $s_3(\beta)$ is a source (in $E_3$). In that case, $s_3^{-1}(\beta_1) = \{\eta\}$, with $r_3^{-1}(r_3(\eta)) = \{\eta\}$, for some $\eta \in E_3^1$. Now, denote $\beta = \beta' \eta$. Then
        \[u_\beta u_\beta^* = u_{\beta'} u_\eta u_\eta^* u_{\beta'}^* = u_{\beta'} p_{r_3(\eta)} u_{\beta'}^* = u_{\beta'} p_{s_3(\beta')} u_{\beta'}^* = u_{\beta'} u_{\beta'}^*.\]
        Since $\beta'$ has even length, we find that $u_\beta u_\beta^* \in \pi(\D(E_1))$. Therefore, $P_1 \D(E_3) P_1 \subseteq \pi(\D(E_1))$. This implies that $\pi(\D(E_1)) = P_1 \D(E_3) P_1$.
        
        A symmetric argument shows that there is a *-isomorphism $\pi_2\colon C^*(E_2) \to P_2 C^*(E_3) P_2$ that is diagonal-preserving and intertwines the gauge action $\gamma^{E_2}$ and $\alpha$.
        We conclude that the triples $(C^*(E_1), \gamma^{E_1} ,\D(E_1))$ and $(C^*(E_2),  \gamma^{E_2}, \D(E_2))$ are Morita equivalent.
        \end{proof}
    
    We are interested in constructing edge potentials for $E_1$ and $E_2$ (and $E_3$), given just the edge potential for $E_1$ or $E_2$. Moreover, we want these potentials to preserve the corresponding weighted gauge actions with respect to the maps $\pi_1$ and $\pi_2$. It turns out that this is equivalent to the following graph-theoretical property.
    
    \begin{definition}
        \label{def:weight-preserving}
        Let $E_1$ and $E_2$ be row-finite graphs that are SSE via $E_3$, and let $f\colon E_1^* \to \Z$, $g\colon E_2^* \to \Z$ and $h\colon E_3^* \to \Z$ be edge functions. We say that $\theta_1$ is \emph{weight-preserving (with respect to $f$ and $h$)} if for all $e \in E_1^1$, $h(\theta_1(e)) = f(e)$. Similarly, we say that $\theta_2$ is \emph{weight-preserving (with respect to $g$ and $h$)} if for all $\overline{e} \in E_2^1$, $h(\theta_2(\overline{e})) = g(\overline{e})$. If both conditions are satisfied, we say that $\theta_1$ and $\theta_2$ are \emph{weight-preserving (with respect to $f, g$ and $h$)}.
    \end{definition}

    \begin{lemma}
        \label{prp:weight-preserving-iff-gauge-action}
        Let $E_1$ and $E_2$ be row-finite graphs that are SSE via $E_3$, 
        choose $\theta_1$ and $\theta_2$ in accordance with~\cref{def:sse}, and let $f\colon E_1^* \to \Z$, $g\colon E_2^* \to \Z$ and $h\colon E_3^* \to \Z$ be edge functions. 
        Then $\theta_1$ and $\theta_2$ are weight-preserving if and only if $\pi_1 \circ \gamma_z^{E_1, f} = \gamma_z^{E_3, h} \circ \pi_1$ and $\pi_2 \circ \gamma_z^{E_2, g} = \gamma_z^{E_3, h} \circ \pi_2$,
        for all $z\in \T$.
    \end{lemma}
    \begin{proof}
        Let $\{s,p\}$ be a universal Cuntz--Krieger family for $C^*(E_1)$.
        First, assume that $\theta_1$ is weight-preserving. Then for all $z \in \T$ and $e \in E_1^1$,
        \[(\pi_1 \circ \gamma^{E_1, f}_z)(s_e) = z^{f(e)} \pi_1(s_e) = z^{h(\theta_1(e))} u_{\theta_1(e)} =  (\gamma^{E_3, h}_z \circ \pi_1)(s_e).\]
        Moreover, for all $z \in \T$ and $v \in E_1^0$,
        \[(\pi_1 \circ \gamma_z^{E_1, f})(p_v) = \pi_1(p_v) = r_v = (\gamma_z^{E_3, h} \circ \pi_1)(p_v).\]
        Hence, $\pi_1 \circ \gamma^{E_1, f} = \gamma^{E_3, h} \circ \pi_1$. Similarly, we also find that $\pi_2 \circ \gamma^{E_2, g} = \gamma^{E_3, h} \circ \pi_2$.

        Now, assume that $\pi_1 \circ \gamma^{E_1, f} = \gamma^{E_3, h} \circ \pi_1$ and $\pi_1 \circ \gamma^{E_2, g} = \gamma^{E_3, h} \circ \pi_2$, and let $e \in E_1^1$. Then for all $z \in \T$,
        \[z^{h(\theta_1(e))} u_{\theta_1(e)} = (\gamma_z^{E_3, h} \circ \pi_1)(s_e) = (\pi_1 \circ \gamma_z^{E_1, f})(s_e) = z^{f(e)} u_{\theta_1(e)}.\]
        So, $h(\theta_1(e)) = f(e)$. Similarly, we also have $h(\theta_2(e)) = g(e)$ for all $e \in E_2^1$. This implies that $\theta_1$ is weight-preserving.
        A similar argument applies to $\theta_2$.
    \end{proof}

    \begin{remark} \label{rmk:only-E3-weight}
        This condition of weight-preserving can be further simplified. If there exist edge functions $f\colon E_1^* \to \Z$ and $h\colon E_3^* \to \Z$ such that $\theta_1$ is weight-preserving, then we can define $g\colon E_2^* \to \Z$ by $g(\overline{e}) = h(\theta_2(\overline{e}))$ so that $\theta_2$ is precisely weight-preserving.
    \end{remark}
    
    We will establish two ways in which we can exploit the structure of SSE to construct the map $g$ given the map $f$. We will later see that these weighted functions will allow us to construct the `natural' edge potential $g$ for $E_2$ for both in-splits and out-splits.
    
    \begin{proposition}
        \label{prp:weight-preserving-with-bijection-E21}
        Let $E_1$ and $E_2$ be row-finite graphs that are SSE via $E_3$, 
        and assume there is a bijection $\phi\colon E_1^1 \to E_{12}^1$ such that for $e \in E_1^1$, if $\theta_1(e) = \eta_2\eta_1$, then $\phi(e) = \eta_2$. 
        Then for any edge function $f\colon E_1^1 \to \Z$, there exist corresponding edge functions $g\colon E_2^1 \to \Z$ and $h\colon E_3^1 \to \Z$ such that $\theta_1$ and $\theta_2$ are weight-preserving.
    \end{proposition}
    \begin{proof}
        Define the map $h\colon E_3^1 \to \Z$ by
        \[h(\eta) = \begin{cases}
            f(\phi^{-1}(\eta)) & \eta \in E_{12}^1 \\
            0 & \eta \in E_{21}^1,
        \end{cases}\]
        for all $\eta\in E_3^1$,
        and $g\colon E_2^1 \to \Z$ by $g(\overline{e}) = h(\iota_2)$, where $\theta_2(e) = \iota_1\iota_2$ for all $e\in E_2^1$.
        We can extend both of these functions to edge functions $g\colon E_3^* \to \Z$ and $h\colon E_2^* \to \Z$. 
        
        Next, let $e \in E_1^1$ with $\theta_1(e) = \eta_2\eta_1$. Then
        \[h(\theta_1(e)) = h(\eta_2) + h(\eta_1) = h(\phi(e)) = f(e).\]
        Similarly, for all $\overline{e} \in E^2$ with $\theta_2(\overline{e}) = \iota_1\iota_2$,
        \[h(\theta_2(\overline{e})) = h(\iota_1) + h(\iota_2) = h(\iota_2) = g(\overline{e}).\]
        So, $\theta_1$ and $\theta_2$ are weight-preserving.
    \end{proof}
    
    \begin{proposition}
        \label{prp:weight-preserving-with-bijection-E12}
        Let $E_1$ and $E_2$ be row-finite graphs that are SSE via $E_3$, 
        and assume there is a with a bijection $\phi\colon E_1^1 \to E_{21}^1$ such that for $e \in E_1^1$, if $\theta_1(e) = \eta_2\eta_1$, then $\phi(e) = \eta_1$. 
        Then for any edge function $f\colon E_1^1 \to \Z$, there exist corresponding edge functions $g\colon E_2^1 \to \Z$ and $h\colon E_3^1 \to \Z$ such that $\theta_1$ and $\theta_2$ are weight-preserving.
    \end{proposition}
    \begin{proof}
        This follows the same way as the proof above, but we define the map $h\colon E_3^1 \to \Z$ as follows:
        \[h(e) = \begin{cases}
            f(\phi^{-1}(\eta)) & \eta \in E_{21}^1 \\
            0 & \eta \in E_{12}^1,
        \end{cases}\]
        for $\eta\in E_3^1$.
    \end{proof}
    
    \section{In-splits} \label{sec:insplits}
    In this section, we will explore in-splits of row-finite graphs and show that they are examples of strong shift equivalences.
    This slightly generalises~\cite[Section 6]{bates2004flow}. 
    We then show that the in-splits induce a *-isomorphism of graph C*-algebras that behaves nicely with respect to weighted gauge actions. 
    This allows us to formulate the question of whether diagonal-preserving *-isomorphims that behave nicely with respect to certain weighted gauge actions in fact characterises when two graph may be connected by in-splits, cf.~\cref{rem:in-split-question}.
    
    This definition of in-splits is from~\cite[Section 3]{bates2004flow}.
    \begin{definition}[In-splits]
        \label{def:in-splits}
        Let $E_1 = (E_1^0, E_1^1, r_1, s_1)$ be a directed graph. For each $v \in E_1^0$ that is not a source, partition the set $r_1^{-1}(v)$ into non-empty sets $\mathcal{E}^v_1, \mathcal{E}^v_2, \dots, \mathcal{E}^v_{m(v)}$ (if $v$ is a source, then $m(v) = 0$), and let $\mathcal{P}$ be the generated partition of $E_1^1$. We define the \emph{in-split graph} $E_2 = (E_2^0, E_2^1, r_2, s_2)$ as follows:
        \begin{align*}
            E_2^0 &= \{\overline{v}_i \mid v \in E_1^0, 1 \leq i \leq m(v)\} \cup \{\overline{v} \mid v \in E_1^0, m(v) = 0\}, \\
            E_2^1 &= \{\overline{e}_i \mid e \in E_1^1, 1 \leq i \leq m(s_1(e))\} \cup \{\overline{e} \mid e \in E_1^1, m(s_1(e)) = 0\},
        \end{align*}
        with the functions $s_2, r_2\colon E_2^1 \to E_2^0$ defined by:
        \[s_2(\overline{e}) = \overline{s_1(e)}, \qquad s_2(\overline{e}_j) = \overline{s_1(e)}_j, \qquad r_2(\overline{e}) = \overline{r_1(e)}_i, \qquad r_2(\overline{e}_j) = \overline{r_1(e)}_i,\]
        with $e \in \mathcal{E}_i^{r_1(e)}$. We say that the in-split is \emph{proper} if for all $v \in E_1^0$ with $r_1^{-1}(v)$ infinite, we have $m(v)$ finite and only one of the partition sets $\mathcal{E}_i^v$ is infinite. 
    \end{definition}
    
    \noindent We first verify that in-splits give examples of strong shift equivalences.
    
    \begin{lemma}
        \label{prp:in-split-SSE}
        Let $E_1$ be a row-finite directed graph and let $E_2$ be a proper in-split of $E_1$. Then there exists a row-finite directed graph $E_3$ such that $E_1$ and $E_2$ are SSE via $E_3$.  Moreover, for any edge function $f\colon E_1^1 \to \Z$, there exist corresponding edge functions $g\colon E_2^1 \to \Z$ and $h\colon E_3^1 \to \Z$ such that $\theta_1$ and $\theta_2$ (chosen in accordance with~\cref{def:sse}) are weight-preserving.
    \end{lemma}
    \begin{proof}
        We construct the graph $E_3$ such that $E_1$ and $E_2$ are SSE via $E_3$. The vertex set $E_3^0$ is composed of vertices $v \in E_1^0$ and $\overline{v} \in E_2^0$. For every vertex $w \in E_1^0$, if $w$ is a source, then we have an edge $\eta \in E_{12}^1$ such that $s_3(\eta) = \overline{w}$ and $r_3(\eta) = w$; otherwise for every $1 \leq i \leq m(v)$, there exists an edge $\eta \in E_{12}^1$ with $s_3(\eta) = \overline{w}_i$ and $r_3(\eta) = w$- this defines a bijection $\phi_1\colon E_2^0 \to E_{12}^1$. Also, for every $f \in E_1^1$, there is an edge $\iota \in E_{21}^1$ such that $s_3(\iota) = s_1(f)$ and $r_3(\iota) = \overline{r_1(f)}_i$, where $f \in \mathcal{E}^{r_1(f)}_i$- this gives a bijection $\phi_2\colon E_1^1 \to E_{21}^1$.

        Now, we define the map $\theta_1\colon E_1^1 \to E_3^2(E_1^0, E_1^0)$ by $\theta_1(e) =  \phi_1(\overline{r_1(e)}_i) \phi_2(e) = \iota_2 \iota_1$, with $e \in \mathcal{E}^{r_1(e)}_i$. This is a valid path of length 2 since $r_3(\iota_1) = \overline{r_1(e)}_i = s_3(\iota_2)$. The edge $\iota_1$ has source $s_1(e)$ and the edge $\iota_2$ has range $r_1(e)$, so $\theta_1$ is source- and range-preserving. Moreover, for $e, f \in E_1^1$ with $\theta_1(e) = \theta_1(f)$, we have $\phi_2(e) = \phi_2(f)$. So, the injectivity of $\phi_2$ implies that $\theta_1$ is injective.

        Also, we claim that $\theta_1$ is surjective. Let $\eta_2\eta_1 \in E_3^2(E_1^0, E_1^0)$. Since $\phi_2$ is surjective, there exists an $e \in E_1^1$ such that $\phi_2(e) = \eta_1$. Let $e$ have source $v$ and range $w$. So, the edge $\eta_1$ must have source $v$ and range $\overline{w}_i$, with $e \in \mathcal{E}^{w}_i$. In that case, $\eta_2 = \phi_1(\overline{w}_i)$. Hence, $\theta_1(e) = \phi_1(\overline{w}_i) \phi_2(e) = \eta_2 \eta_1$. So, $\theta_1$ is surjective. This implies that $\theta_1$ is a bijection.

        Next, we define the map $\theta_2\colon E_2^1 \to E_3^2(E_2^0, E_2^0)$ by $\theta_2(\overline{e}) = \phi_2(e) \phi_1(\overline{s_1(e)}) = \iota_1\iota_2$ and $\theta_2(\overline{e}_j) = \phi_2(e) \phi_1(\overline{s_1(e)}_j) = \iota_1\iota_2$. This is a valid path since $r_3(\iota_2) = s_1(e) = s_3(\iota_1)$. The edge $\iota_2$ has source $\overline{s_1(e)} = s_2(\overline{e})$ or $\overline{s_1(e)}_i = s_2(\overline{e}_i)$, and $\iota_1$ has range $\overline{r_1(e)}_i$, with $e \in \mathcal{E}_i^{r_1(e)}$. So, $\theta_2$ is source- and range-preserving. 
        
        Now, we show that $\theta_2$ is injective. Let $\overline{e}, \overline{f} \in E_2^1$ with $\theta_2(\overline{e}) = \theta_2(\overline{f})$. Since $\phi_2$ is injective, we find that $e = f$. So, $\overline{e} = \overline{f}$. Next, let $\overline{e}_j, \overline{f}_k \in E_2^1$ with $\theta_2(\overline{e}_j) = \theta_2(\overline{f}_k)$. Like above, since $\phi_2$ is injective, we have $e = f$. Moreover, since $\phi_1$ is injective, we find that $\overline{s_1(e)}_j = \overline{s_1(f)}_k$, meaning that $j = k$. So, $\overline{e}_j = \overline{f}_k$. Hence, $\theta_2$ is injective.
        
        Next, we claim that $\theta_2$ is surjective. So, let $\iota_1\iota_2 \in E_3^2(E_2^0, E_2^0)$. We can find an $e \in E_1^0$ such that $\iota_1 = \phi_2(e)$. If $\iota_2 = \phi_1(\overline{v})$, then $s_1(e) = v$ is a source. So, $\theta_2(\overline{e}) = \phi_2(e) \phi_1(\overline{s_1(e)}) = \iota_1\iota_2$. Instead, if $\iota_2 = \phi_1(\overline{v}_i)$, then $\theta_2(\overline{e}_i) =  \phi_2(e) \phi_1(\overline{s_1(e)}_i) = \iota_1 \iota_2$. So, $\theta_2$ is surjective. This implies that $\theta_2$ is a bijection.

        Now, let $\overline{v} \in E_3^0$ be a source (in $E_3$). By $\phi_1$, we know that for all $v \in E_1^0$, $v$ is not a source (in $E_1$). So, we must have $\overline{v} \in E_2^0$. Since $\overline{v}$ is a source, we know that $m(v) = 0$, meaning that $s_3^{-1}(\overline{v}) = \{\varphi_1(\overline{v})\}$. Since $r_3(\varphi_1(\overline{v})) \in E_1^0$ and $\varphi_1$ is a bijection, we find that $r_3^{-1}(r_3(\varphi_1(\overline{v}))) = \{\varphi_1(\overline{v})\}$. Thus, $E_1$ and $E_2$ are SSE via $E_3$.
        
        Finally, consider the map $\phi_2\colon E_1^1 \to E_{21}^1$. This map satisfies the condition that for all $e \in E_1^1$, if $\theta_1(e) = \eta_2 \eta_1$, then $\eta_1 = \phi_2(e)$ by construction. Hence, by \Cref{prp:weight-preserving-with-bijection-E12}, we can construct the maps $g\colon E_2^* \to \Z$ and $h\colon E_3^* \to \Z$ give an edge map $f\colon E_1^* \to \Z$ that is weight-preserving accordingly.
    \end{proof}
    
    It now follows from~\cref{thm:morita-equivalence-sse} that for a row-finite graph and its in-split graph,
    the graph C*-algebras are Morita equivalent in a diagonal-preserving and gauge-equivariant way.
    However, in the case of in-splits the graph C*-algebras are in fact *-isomorphic and we can say more. 
    The theorem below should be compared to~\cite[Theorem 3.2]{bates2004flow} (which proves that the graph C*-algebras are *-isomorphic),~\cite[Theorem 1]{matsumoto_2022} (which applies to strongly connected finite graphs),~\cite[Theorem 3.2]{eilers2019refined} (which applies to graphs with finitely many vertices), and finally~\cite[Corollary 3.3]{ABCE} (which is from the point of view of conjugacy of dynamical systems).
    
    \begin{theorem}
        \label{thm:gauge-action-*-iso-in-split}
        Let $E_1$ be a row-finite directed graph, and let $E_2$ be a proper in-split graph of $E_1$. 
        Then there exists a *-isomorphism $\pi\colon C^*(E_1) \to C^*(E_2)$ satisfying $\pi_1(\D(E_1)) = \D(E_2)$ and, moreover, 
        for any edge function $f\colon E_1^1 \to \Z$, there exists an edge function $g\colon E_2^1 \to \Z$ such that $\pi \circ \gamma_z^{E_1, f} = \gamma_z^{E_2, g} \circ \pi$ for all $z \in \T$.
    \end{theorem}
    \begin{proof}
    Let $\{s,p\}$ and $\{t,q\}$ be universal Cuntz--Krieger families for $C^*(E_1)$ and $C^*(E_2$), respectively.
        The isomorphism result is given in \cite[Theorem~3.1]{bates2004flow}.
        An explicit *-isomorphism is given as $\pi\colon C^*(E_1) \to C^*(E_2)$ where 
        \[\pi(s_e) = \begin{cases}
            t_{\overline{e}} & m(s(e)) = 0 \\
            \sum_{i=1}^{m(s(e))} t_{\overline{e}_i} & \textrm{otherwise},
        \end{cases} \qquad \pi(p_v) = \begin{cases}
            q_{\overline{v}} & m(v) = 0 \\
            \sum_{i=1}^{m(v)} q_{\overline{v}_i} & \textrm{otherwise}
        \end{cases}\]
        for all $e\in E_1^1$ and $v\in E_1^0$.
        The fact that $\pi(\D(E_1)) = \D(E_2)$ is given in~\cite[Theorem~3.2]{eilers2019refined}.
        
        We now show the equivariance result.
        Let $e \in E_1^1$. 
        If $s_1(e)$ is a source, then define $g(\overline{e}) = f(e)$. 
        Otherwise, for $1 \leq i \leq m(s_1(e))$, define $g(\overline{e}_i) = f(e)$. 
        This can be extended into an edge function $g\colon E_2^* \to \Z$. 
        
        Now, let $z \in \T$. We claim that $\pi \circ \gamma_z^{E_1, f} = \gamma_z^{E_2, g} \circ \pi$ for all $z\in \T$. 
        First, let $v \in E_1^0$. If $v$ is a source, then we have
        \[
        (\gamma^{E_2, g}_z \circ \pi)(p_v) = \gamma_z^{E_2, g}(q_{\overline{v}}) = q_{\overline{v}} = \pi(p_v) = (\pi \circ \gamma_z^{E_1, f})(p_v).
        \]
        Instead, if $v$ is not a source, then
        \[
        (\gamma^{E_2, g}_z \circ \pi)(p_v) = \sum_{i=1}^{m(v)} \gamma_z^{E_2, g}(q_{\overline{v}_i}) = \sum_{i=1}^{m(v)} p_{\overline{v}_i} = P_v = (\pi \circ \gamma_z^{E_1, f})(p_v).
        \]
        Next, let $e \in E_1^1$. If $s_1(e)$ is a source, we have
        \[
        (\gamma^{E_2, g}_z \circ \pi)(s_e) = \gamma_z^{E_2, g}(t_{\overline{e}}) = z^{g(\overline{e})} t_{\overline{e}} = z^{f(e)} \pi(s_e) = (\pi \circ \gamma_z^{E_1, f})(s_e).
        \]
        Instead, if $s_1(e)$ is not a source, then we have
        \[
        (\gamma^{E_2, g}_z \circ \pi)(s_e) = \sum_{i=1}^{m(s_1(e))} z^{g(\overline{e}_i)} t_{\overline{e}_i} 
        = z^{f(e)}  \sum_{i=1}^{m(s_1(e))} t_{\overline{e}_i} 
        = (\pi \circ \gamma^{E_1, f})(s_e).
        \]
        So, we have $\gamma_z^{E_2, g} \circ \pi = \pi \circ \gamma_z^{E_1, f}$ with respect to the generators. 
        Hence, $\gamma_z^{E_2, g} \circ \pi = \pi \circ \gamma_z^{E_1, f}$ for all $z\in \T$.
    \end{proof}
    
    \begin{example}
        Let $E_1$ be the following weighted graph:
        \begin{figure}[H]
            \centering
            \begin{tikzpicture}
                \node (v) at (0, 0) {$v$};
                \node (w) at (1.5, 0) {$w$};
                
                \arrow (v) edge [loop left] node {1} (v);
                \arrow (w) -- node[above] {2} (v);
            \end{tikzpicture}
        \end{figure}
        \noindent After a proper in-split at the vertex $v$, we obtain the following weighted graph $E_2$:
        \begin{figure}[H]
            \centering
            \begin{tikzpicture}
                \node (v1) at (-1.5, 0) {$\overline{v}_1$};
                \node (v2) at (0, 0) {$\overline{v}_2$};
                \node (w) at (1.5, 0) {$\overline{w}$};
                
                \arrow (v1) edge [loop left] node {1} (v1);
                \arrow (v2) -- node[above] {1} (v1);
                \arrow (w) -- node[above] {2} (v2);
            \end{tikzpicture}
        \end{figure}
    \end{example}
    
    \begin{example}
        \label{rmk:*-iso-reverse}
        \Cref{thm:gauge-action-*-iso-in-split} allows us to construct the edge function $g$ on $E_2$ given an edge function $f$ on $E_1$. The converse is not normally true. For instance, consider the following pair of graphs $E_1$ and $E_2$:
        
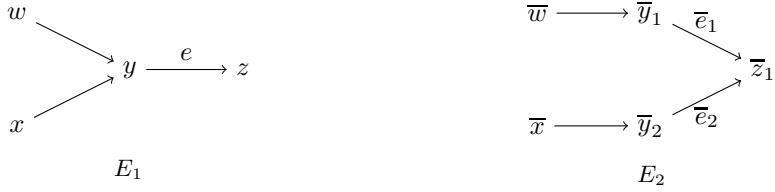
\begin{figure}[H]
            \centering
            \begin{subfigure}{0.45\textwidth}
                \centering
                \begin{tikzpicture}
                    \node (w) at (0, 0) {$w$};
                    \node (x) at (0, -1.5) {$x$};
                    \node (y) at (1.5, -0.75) {$y$};
                    \node (z) at (3, -0.75) {$z$};
    
                    \arrow (w) -- (y);
                    \arrow (x) -- (y);
                    \arrow (y) -- node[above] {$e$} (z);
                \end{tikzpicture}
                \caption{$E_1$}
            \end{subfigure}
            \hfill
            \begin{subfigure}{0.45\textwidth}
                \centering
                \begin{tikzpicture}
                    \node (W) at (0, 0) {$\overline{w}$};
                    \node (X) at (0, -1.5) {$\overline{x}$};
                    \node (Y1) at (1.5, 0) {$\overline{y}_1$};
                    \node (Y2) at (1.5, -1.5) {$\overline{y}_2$};
                    \node (Z) at (3, -0.75) {$\overline{z}_1$};
    
                    \arrow (W) -- (Y1);
                    \arrow (X) -- (Y2);
                    \arrow (Y1) -- node[above] {$\overline{e}_1$} (Z);
                    \arrow (Y2) -- node[below] {$\overline{e}_2$} (Z);
                \end{tikzpicture}
                \caption{$E_2$}
            \end{subfigure}
        \end{figure}
        \noindent Let $g\colon E_2^1 \to \Z$ be an edge function on $E_2$ such that $g(\overline{e}_1) \neq g(\overline{e}_2)$, and let $f\colon E_1^1 \to \Z$ be an arbitrary edge function. 
        Then
        \begin{align*}
            (\gamma_z^{E_2, g} \circ \pi)(s_e) &= z^{g(\overline{e}_1)} t_{\overline{e}_1} + z^{g(\overline{e}_2)} t_{\overline{e}_2}, \\
            (\pi \circ \gamma^{E_1, f})(s_e) &= z^{f(e)} t_{\overline{e}_1} + z^{f(e)} t_{\overline{e}_2}.
        \end{align*}
        \noindent Since $g(\overline{e}_1) \neq g(\overline{e}_2)$, there cannot exist an edge function $f$ on $E_1$ such that $(\gamma_z^{E_2, g} \circ \pi)(s_e) = (\pi \circ \gamma^{E_1, f})(s_e)$. Nonetheless, it might be possible to construct a weight function such that $\gamma^{E_2, g} \circ \pi = \pi \circ \gamma^{E_1, f}$.
    \end{example}
    
    \begin{remark}
    \label{rem:in-split-question}
        We may now ask whether the converse of~\cref{thm:gauge-action-*-iso-in-split} holds.
        Namely, if $\pi\colon C^*(E_1) \to C^*(E_2)$ is a *-isomorphism of graph $C^*$-algebras satisfying the conclusion of the previous theorem,
        does it follow that $E_1$ and $E_2$ can be connected by in-splits (and inverse)? 
        A solution to this problem might involve some generalisation of in-splits.
    \end{remark}
    
    \section{Out-splits} \label{sec:outsplits}
    In this section, we will explore a dual notion to in-splits called out-splits. 
    This too will lead to a strong shift equivalence. 
    
    This definition is from~\cite[Section 5]{bates2004flow}.
    \begin{definition}[Out-splits]
        \label{def:out-splits}
        Let $E_1$ be a directed graph, and for each $v \in E_1^0$ not source, partition the set $s_1^{-1}(v)$ into sets $\mathcal{E}_v^1, \mathcal{E}_v^2, \dots, \mathcal{E}_v^{m(v)}$ (if $v$ is a source, then $m(v) = 0$), and let $\mathcal{P}$ be the generated partition of $E_1^1$ (including $s^{-1}(v)$ for $v$ source, without partitioning). We define the \emph{out-split graph} $E_2$ as follows:
        \begin{align*}
            E_2^0 &= \{\overline{v}^i \mid v \in E_1^0, 1 \leq i \leq m(v)\} \cup \{\overline{v} \mid v \in E_1^0, m(v) = 0\}, \\
            E_2^1 &= \{\overline{e}^i \mid e \in E_1^1, 1 \leq i \leq m(r_1(e))\} \cup \{\overline{e} \mid e \in E_1^1, m(r_1(e)) = 0\}.
        \end{align*}
        The functions $s_2, r_2\colon E_2^1 \to E_2^0$ are given by:
        \[s_2(\overline{e}) = \overline{s_1(e)}^i, \qquad s_2(\overline{e}^j) = \overline{s_1(e)}^i, \qquad r_2(\overline{e}) = \overline{r_1(e)}, \qquad r_2(\overline{e}^j) = \overline{r_1(e)}^j,\]
        with $e \in \mathcal{E}^i_{r_1(e)}$. We say that the out-split is \emph{proper} if for all $v \in E_1^0$ with $r_1^{-1}(v)$ infinite, we have $m(v) = 1$, i.e. infinite receivers cannot be split \cite[Section~5]{bates2004flow}. 
    \end{definition}
    
    For out-splits, there need not be a *-isomorphism between $C^*(E_1)$ and $C^*(E_2)$. Nonetheless, we can show that the C*-algebras are Morita equivalent, like in the case of in-splits.
    \begin{lemma}
        \label{lem:out-split-sse}
        Let $E_1$ be a row-finite directed graph and let $E_2$ be a proper out-split of $E_1$. 
        Then there exists a row-finite directed graph $E_3$ such that $E_1$ and $E_2$ are SSE via $E_3$.
        Moreover, for any edge function $f\colon E_1^1 \to \Z$, there exist corresponding edge functions $g\colon E_2^1 \to \Z$ and $h\colon E_3^1 \to \Z$ 
        such that $\theta_1$ and $\theta_2$ (chosen according to~\cref{def:sse} are weight-preserving.
    \end{lemma}
    \begin{proof}
        We construct the graph $E_3$ via which $E_1$ and $E_2$ are SSE. We define $E_3^0 := E_1^0 \cup E_2^0$. For every vertex $w \in E_1^0$, if $m(w) = 0$, then we have an edge $\eta \in E_{21}^1$ such that $s_3(\eta) = w$ and $r_3(\eta) = \overline{w}$; otherwise for every $1 \leq i \leq m(v)$, there exists an edge $\eta \in E_{21}^1$ with $s_3(\eta) = w$ and $r_3(\eta) = \overline{w}^i$- this defines a bijection $\phi_1\colon E_2^0 \to E_{21}^1$. Also, for every $e \in E_1^1$, there is an edge $\iota \in E_{12}^1$ such that $s_3(\iota) = \overline{s_1(e)}^i$ and $r_3(\iota) = r_1(e)$, where $e \in \mathcal{E}_{s_1(e)}^i$- this gives a bijection $\phi_2\colon E_1^1 \to E_{12}^1$.

        Now, we can define the map $\theta_1\colon E_1^1 \to E_3^2(E_1^0, E_1^0)$ by $\theta_1(e) = \phi_2(e) \phi_1(\overline{s_1(e)}^i) = \iota_2\iota_1$, with $e \in \mathcal{E}_{s_1(e)}^i$, and $\theta_2\colon E_2^1 \to E_3^2(E_2^0, E_2^0)$ by $\theta_2(\overline{e}) = \phi_1(\overline{r_1(e)}) \phi_2(e) = \iota_1\iota_2$ and $\theta_2(\overline{e}^j) = \phi_1(\overline{r_1(e)}^j) \phi_2(e)$. Using the same argument as for in-splits, we find that $E_1$ and $E_2$ are SSE via $E_3$.
        
        Finally, consider the map $\phi_2\colon E_1^1 \to E_{12}^1$. This map satisfies the condition that for all $e \in E_1^1$, if $\theta_1(e) = \eta_2 \eta_1$, then $\eta_2 = \phi_2(e)$ by construction. Hence, by \Cref{prp:weight-preserving-with-bijection-E21}, we can construct the maps $g$ and $h$ accordingly.
    \end{proof}

    The corollary below now follows from~\cref{thm:morita-equivalence-sse} and~\cref{lem:out-split-sse}.
    This result should be compared to~\cite[Theorem 1.3]{Matsumoto2017}.
    \begin{corollary}
    Let $E_1$ be a row-finite directed graph and let $E_2$ be a proper out-split of $E_1$.
    If $f\colon E_1^1\to \Z$ is any edge function then there exists an edge function $g\colon E_2^1\to \Z$ such that the triples
    $(C^*(E_1), \gamma^{E_1,f}, \D(E_1))$ and $(C^*(E_2), \gamma^{E_2,g}, \D(E_2))$ are Morita equivalent.
    \end{corollary}

    \begin{example}
        Let $E_1$ be the following weighted graph:
        
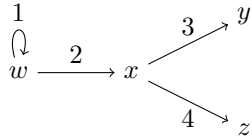
\begin{figure}[H]
            \centering
            \begin{tikzpicture}
                \node (w) at (0, 0) {$w$};
                \node (x) at (1.5, 0) {$x$};
                \node (y) at (3, 0.75) {$y$};
                \node (z) at (3, -0.75) {$z$};

                \arrow (w) edge[loop above] node {$1$} (w);
                \arrow (w) -- node[above] {$2$} (x);
                \arrow (x) -- node[above] {$3$} (y);
                \arrow (x) -- node[below] {$4$} (z);
            \end{tikzpicture}
        \end{figure}        
        \noindent Then the weight on a proper out-split graph of $E_1$, $E_2$, is the following:
        
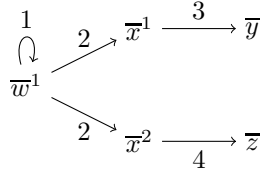
\begin{figure}[H]
            \centering
            \begin{tikzpicture}
                \node (w) at (0, 0) {$\overline{w}^1$};
                \node (x1) at (1.5, 0.75) {$\overline{x}^1$};
                \node (x2) at (1.5, -0.75) {$\overline{x}^2$};
                \node (y) at (3, 0.75) {$\overline{y}$};
                \node (z) at (3, -0.75) {$\overline{z}$};
                
                \arrow (w) edge[loop above] node {$1$} (w);
                \arrow (w) -- node[above] {$2$} (x1);
                \arrow (w) -- node[below] {$2$} (x2);
                \arrow (x1) -- node[above] {$3$} (y);
                \arrow (x2) -- node[below] {$4$} (z);
            \end{tikzpicture}
        \end{figure}
        \noindent And the weight on $E_3$ is the following:
        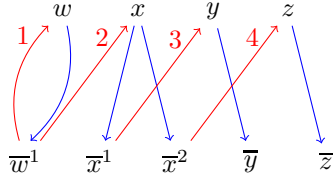
\begin{figure}[H]
            \centering
            \begin{tikzpicture}
                \node (w) at (0, 2) {$w$};
                \node (x) at (1, 2) {$x$};
                \node (y) at (2, 2) {$y$};
                \node (z) at (3, 2) {$z$};

                \node (W) at (-0.5, 0) {$\overline{w}^1$};
                \node (X1) at (0.5, 0) {$\overline{x}^1$};
                \node (X2) at (1.5, 0) {$\overline{x}^2$};
                \node (Y) at (2.5, 0) {$\overline{y}$};
                \node (Z) at (3.5, 0) {$\overline{z}$};

                \bluearrow (w) edge[bend left] (W.75);
                \bluearrow (x) -- (X1);
                \bluearrow (x) -- (X2);
                \bluearrow (y) -- (Y);
                \bluearrow (z) -- (Z);

                \redarrow (W) edge[bend left] node[left, pos=0.9] {$1$} (w);
                \redarrow (W) -- node[left, pos=0.9] {$2$} (x);
                \redarrow (X1) -- node[left, pos=0.9] {$3$} (y);
                \redarrow (X2) -- node[left, pos=0.9] {$4$} (z);
            \end{tikzpicture}
            \caption{The directed graph $E_3$ via which $E_1$ and $E_2$ are SSE. The blue arrows have weight 0, and the weight for the red arrows is given --- this is directly taken using the bijection $\phi_2\colon E_1^1 \to E_{12}^1$.}
        \end{figure}
    \end{example}
    
    \begin{remark}
        \label{rmk:SSE-reverse}
        Like for the isomorphism $\pi$ in the case of in-splits, it is only typically possible to construct an edge potential from the original graph $E_1$ to its in-split/out-split graph $E_2$ (and $E_3$). We cannot use the same idea as the one above to construct an edge function on $E_1$ given an edge function on $E_2$. In particular, there need not be a bijection from $E_2^1$ to either $E_{12}^1$ or $E_{21}^1$. 
        For instance, let $E_1$, $E_2$ and $E_3$ be the following graphs:
        
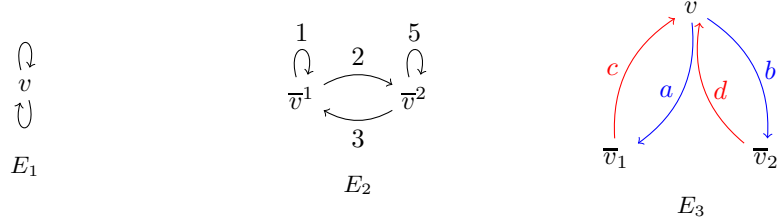
\begin{figure}[H]
            \centering
            \begin{subfigure}{0.30\textwidth}
                \centering
                \begin{tikzpicture}
                    \node (v) at (0, 0) {$v$};
                    
                    \arrow (v) edge [loop above] (v);
                    \arrow (v) edge [loop below] (v);
                \end{tikzpicture}
                \caption{$E_1$}
            \end{subfigure}
            \hfill
            \begin{subfigure}{0.30\textwidth}
                \centering
                \begin{tikzpicture}
                    \node (V1) at (0, 0) {$\overline{v}^1$};
                    \node (V2) at (1.5, 0) {$\overline{v}^2$};
                    
                    \arrow (V1) edge [loop above] node {$1$} (V1);
                    \arrow (V1) edge [bend left] node[above] {$2$} (V2);
                    \arrow (V2) edge [bend left] node[below] {$3$} (V1);
                    \arrow (V2) edge [loop above] node {$5$} (V2);
                \end{tikzpicture}
                \caption{$E_2$}
            \end{subfigure}
            \hfill
            \begin{subfigure}{0.30\textwidth}
                \centering
                \begin{tikzpicture}
                    \node (v) at (0, 2) {$v$};
    
                    \node (V1) at (-1, 0) {$\overline{v}_1$};
                    \node (V2) at (1, 0) {$\overline{v}_2$};
    
                    \bluearrow (v) edge[bend left] node[left] {$a$} (V1);
                    \bluearrow (v) edge[bend left] node[right] {$b$} (V2);
                    
                    \redarrow (V1) edge[bend left] node[left] {$c$} (v);
                    \redarrow (V2) edge[bend left] node[right] {$d$} (v.-60);
                \end{tikzpicture}
                \caption{$E_3$}
            \end{subfigure}
        \end{figure}
        \noindent The weights on $E_2^1$ give rise to a specific edge function on $E_2$, and similarly an arbitrary edge function on $E_3$. Although $E_1$ and $E_2$ are SSE via $E_3$, for the edge function on $E_2$, we cannot construct an edge function on $E_3$, i.e. the values $a, b, c, d$ cannot be defined. If this were possible, we would require $\theta_2$ to be weight-preserving. That is, we need the following system of equations to have a solution:
        \[\begin{cases}
            c + a = 1 \\
            c + b = 2 \\
            d + a = 3 \\
            d + b = 5.
        \end{cases}\]
        However, 
        \[d + b = (c + b) - (c + a) + (d + a) = 4.\]
        This implies that the system of equations has no solutions --- we cannot construct an edge function $E_3$. Note that it is possible to construct edge functions on $E_3$ by slightly changing the edge function on $E_2$ --- if the loop on $v_2$ has weight $4$.
    \end{remark}
    
    \begin{remark}
    From a row-finite graph $E$ we may construct the triple consisting of the graph C*-algebra $(C^*(E)$, its canonical gauge action $\gamma^E$, and the diagonal subalgebra $\D(E))$. 
    It would be interesting to clarify the connection of Morita equivalence of such triples to the characterisation of strong shift equivalence for finite graphs with no sinks or sources given in~\cite{Carlsen-Rout}.
    At the moment, we do not know if strong shift equivalence for row-finite graphs is generated by in-splits and out-splits
    although it seems that some kind of generalised notion of splittings is needed for this to be case.
    The weighted gauge actions play an important role in the characterisations of conjugacy in~\cite{matsumoto_2022,ABCE}
    and this might also be the case for strong shift equivalence.
    \end{remark}

    \bibliographystyle{alpha}
    \bibliography{citations}

\newcommand{\etalchar}[1]{$^{#1}$}
\begin{thebibliography}{ABCE22}

\bibitem[ABC{\etalchar{+}}]{ABCEW}
Becky Armstrong, Kevin~Aguyar Brix, Toke~Meier Carlsen, S{\o}ren Eilers, and
  Jeroen Winkel.
\newblock In preparation.

\bibitem[ABCE22]{ABCE}
Becky Armstrong, Kevin~Aguyar Brix, Toke~Meier Carlsen, and S{\o}ren Eilers.
\newblock Conjugacy of local homeomorphisms via groupoids and {C}*-algebras.
\newblock {\em Ergodic Theory and Dynamical Systems}, pages 1--22, 2022.

\bibitem[Ash96]{Ashton}
B.~Ashton.
\newblock Morita equivalence of graph {C}*-algebras.
\newblock {\em Honours thesis}, The University of Newcastle (NSW), 1996.

\bibitem[Bat02]{bates2002applications}
Teresa Bates.
\newblock Applications of the gauge-invariant uniqueness theorem for graph
  algebras.
\newblock {\em Bulletin of the Australian Mathematical Society}, 66(1):57--67,
  2002.

\bibitem[BP04]{bates2004flow}
Teresa Bates and David Pask.
\newblock Flow equivalence of graph algebras.
\newblock {\em Ergodic Theory and Dynamical Systems}, 24(2):367--382, 2004.

\bibitem[CK80]{Cuntz-Krieger}
Joachim Cuntz and Wolfgang Krieger.
\newblock A class of {C}*-algebras and topological {M}arkov chains.
\newblock {\em Invent. Math.}, 56(3):251--268, 1980.

\bibitem[Com84]{Combes}
F.~Combes.
\newblock Crossed products and {M}orita equivalence.
\newblock {\em Proc. London Math. Soc. (3)}, 49(2):289--306, 1984.

\bibitem[CR17]{Carlsen-Rout}
Toke~Meier Carlsen and James Rout.
\newblock Diagonal-preserving gauge-invariant isomorphisms of graph
  {C}*-algebras.
\newblock {\em Journal of Functional Analysis}, 273(9):2981--2993, 2017.

\bibitem[ER19]{eilers2019refined}
S{\o}ren Eilers and Efren Ruiz.
\newblock Refined moves for structure-preserving isomorphism of graph
  {C}*-algebras.
\newblock {\em arXiv preprint arXiv:1908.03714}, 2019.

\bibitem[LM21]{Lind-Marcus}
Douglas Lind and Brian Marcus.
\newblock {\em An introduction to symbolic dynamics and coding}.
\newblock Cambridge Mathematical Library. Cambridge University Press,
  Cambridge, 2021.
\newblock Second edition [of 1369092].

\bibitem[Mat17]{Matsumoto2017}
Kengo Matsumoto.
\newblock Topological conjugacy of topological {M}arkov shifts and
  {C}untz--{K}rieger algebras.
\newblock {\em Doc. Math.}, 22:873--915, 2017.

\bibitem[Mat22]{matsumoto_2022}
Kengo Matsumoto.
\newblock On one-sided topological conjugacy of topological markov shifts and
  gauge actions on cuntz-–krieger algebras.
\newblock {\em Ergodic Theory and Dynamical Systems}, 42(8):2575–2582, 2022.

\bibitem[Rae05]{raeburn2005graph}
Iain Raeburn.
\newblock {\em Graph algebras}.
\newblock Number 103. American Mathematical Soc., 2005.

\bibitem[Wil73]{Williams}
R.~F. Williams.
\newblock Classification of subshifts of finite type.
\newblock {\em Ann. of Math. (2)}, 98:120--153; errata, ibid. (2) 99 (1974),
  380--381, 1973.

\end{thebibliography}
\end{document}